\documentclass[12pt]{amsart}
\usepackage{amsfonts}
\usepackage{tipa}
\usepackage{amssymb}
\usepackage{mathrsfs}
\usepackage{amsthm,amsmath,amssymb}
\usepackage{txfonts}
\usepackage{graphicx}
\usepackage{geometry}
\usepackage{graphics}
\usepackage{caption}

\usepackage{paralist}
\usepackage[misc]{ifsym}
\usepackage{epsfig} 
\usepackage{epstopdf} 

\usepackage[colorlinks=true]{hyperref}

\hypersetup{urlcolor=red, citecolor=blue}
\theoremstyle{plain}
\newtheorem{theorem}{Theorem}[section]

\newtheorem{lemma}[theorem]{Lemma}
\newtheorem{proposition}[theorem]{Proposition}

\theoremstyle{definition}

\newtheorem{remark}[theorem]{Remark}

\allowdisplaybreaks[4]

\begin{document}
	
\title{Monotonicity of positive solutions for an indefinite logarithmic Laplacian equation}


\author{Baiyu Liu and Shasha Xu}	

\address[B. Liu]{School of Mathematics and Physics\\
	University of Science and Technology Beijing \\
	30 Xueyuan Road, Haidian District
	Beijing, 100083\\
	P.R. China}
\email{liubymath@gmail.com, liuby@ustb.edu.cn}
\address[S. Xu]{School of Mathematics and Physics\\
	University of Science and Technology Beijing \\
	30 Xueyuan Road, Haidian District
	Beijing, 100083\\
	P.R. China}
\email{m202110717@xs.ustb.edu.cn}

\maketitle

\begin{abstract}
In this paper, we investigate a nonlocal equation involving the logarithmic Laplacian with indefinite nonlinearities:
\begin{equation*}
\left\{
\begin{array}{ll}
L_\Delta u(x)=a(x_n)f(u), & x\in\Omega, \\
u(x)=0,& x\in \mathbb{R}^n\backslash\Omega.
\end{array}
\right.
\end{equation*}
Here, $\Omega$ represents a Lipschitz coercive epigraph. To achieve our objectives, we develop a boundary estimate for antisymmetric functions, enabling us to establish the monotonicity and nonexistence of bounded positive solutions for the above problem using the direct method of moving planes.
\end{abstract}

\maketitle

\section{Introduction}

The logarithmic Laplacian $L_\Delta$ is a nonlocal pseudo-differential operator \cite{CLL1}, assuming the form  
\begin{equation}
	\label{def:logl}
	\begin{aligned}
		L_\Delta u(x)&=
		(-\Delta)^{L}u(x)+\rho_n u(x)\\
		&=C_nP.V.\int_{\mathbb{R}^{n}}\frac{{u(x)1_{B_{1}(x)}}(y)-u(y)}{|x-y|^n}dy+\rho_n u(x)\\
		&=C_{n} P.V.\int_{{B}_{1}(x)}{\frac{u(x)-u(y)}{|x-y|^{n}}}dy+\int_{\mathbb{R}^n\backslash{B_1(x)}}{\frac{-u(y)}{|x-y|^{n}}}dy+\rho_n u(x),
	\end{aligned}
\end{equation}
where $C_{n}:=\pi^{-n/2}\Gamma(n/2)=\frac{2}{|S^{n-1}|}$ is a normalization positive constant, $P.V.$ stands for the Cauchy principal value. 
$\rho_n=2\ln 2+\psi(n/2)-\gamma$, where $\Gamma$ is the Gamma function, $\psi=\Gamma'/\Gamma$ is the Digamma function and $\gamma=-\Gamma'(1)$ is the Euler Macheroni constant. 
Throughout the paper, we shall always require $u\in L_{0}\cap C^{1,1}_{loc}(\mathbb{R}^n)$, in which
$$
L_0:=\{u:\mathbb{R}^n\to R~|~\int_{\mathbb{R}^{n}}\frac{|u(x)|}{1+|x|^{n}}dx<+\infty\}.
$$
Then the singular integral on right hand side of \eqref{def:logl} will make sense. In this paper our aim is to study the monotonicity properties for solutions of the following problem involving logarithmic Laplacian and indefinite nonlinearities:
\begin{equation}
	\label{sec:six}
	\left\{
	\begin{array}{ll}
		L_\Delta u(x)=a(x_n)f(u), & x\in\Omega, \\
		u(x)=0,& x\in \mathbb{R}^n\backslash\Omega.
	\end{array}
	\right.
\end{equation}
Here 
\begin{equation}
	\label{def:epig}
	\Omega=\{x=(x_1,\dots,x_{n-1},x_n)=(x',x_n)\in \mathbb{R}^n \,|\, x_n>\varphi(x')\}
\end{equation} 
is a Lipschitz coercive epigraph, i.e. $\varphi: \mathbb{R}^{n-1}\to \mathbb{R}$ satisfying 
\begin{equation}
	\label{asp:varphi}
	\varphi(x')\ \textrm{is Lipschitz continuous and}
	\lim\limits_{|x'|\to\infty}\varphi(x')=+\infty.
\end{equation}
Define
$l:=\inf\limits_{x'\in \mathbb{R}^{n-1}}\varphi(x').$

In recent years, there has been a remarkable surge of interest in utilizing fractional order operators, including the fractional Laplacian, the fractional p-Laplacian, and the logarithmic Laplacian, to model a diverse array of physical phenomena \cite{Caf2010, Con2006}. This growing fascination is fueled by the profound impact of fractional calculus in various fields, driven by crucial practical applications and groundbreaking advancements in comprehending non-local phenomena. These applications span across diverse disciplines, such as conformal geometry \cite{Chang}, probability and finance \cite{Ber1996, Cab2010}, stratified materials \cite{Sav}, and numerous others.

Recall that for $s\in (0,1)$, the fractional Laplacian $(-\Delta )^s$ can be written as a singular integral operator defined by (see \cite{chenbook}).
$$
(-\Delta)^s u(x)=C_n P.V.\int_{\mathbb{R}^{n}}\frac{u(x)-u(y)}{|x-y|^{n+2s}}dy.
$$
It should be pointed out that the fractional Laplacian can also be defined equivalently through Caffarelli and Silvestre's extension method \cite{CLL4}.

The following fractional equation with indefinite nonlinearites
\begin{equation}
	\label{pro:indef1}
	(-\Delta)^s u(x)=x_1 u^p(x), \quad x\in \mathbb{R}^n
\end{equation}
has been the subject of investigation by several authors in recent years. Notably, for $1/2\leq s <1$ and $1< p < (n + 2s)/(n-2s)$, Chen and Zhu \cite{CLL7} established the nonexistence of positive bounded solutions to equation \eqref{pro:indef1} through the application of extension method.
Subsequently, Chen, Li, and Li \cite{CLL8} applied the direct method of moving planes, instead of extension method, and achieved an improved result by extending the range of $s$ from $[1/2, 1)$ to $(0, 1)$.
More recently, the method of moving planes was utilized by Chen, Li, and Zhu \cite{CLL19} to derive the nonexistence of positive solutions for the equation:
\begin{equation*}
	\label{pro:indef2}
	(-\Delta)^s u(x)=a(x_1) f(u), \quad x\in \mathbb{R}^n
\end{equation*}
with $0 <s<1$, subject to certain appropriate assumptions on $a(x_1)$ and $f(u)$.
For further literature on the methods of moving planes and their diverse applications, interested readers are referred to \cite{H, CLL20, CLL10, CLL11, CLL13, Liu2016, Dai}, and their respective references.

In recent years, significant progress has been made in studying the monotonicity of positive solutions in epigraphs, with several authors contributing to this field. Notably, Esteban and Lions \cite{Est1982} investigated the case of a coercive Lipschitz epigraph defined as in \eqref{def:epig} \eqref{asp:varphi}. Utilizing the method of moving planes, they demonstrated that the positive bounded solution of the following elliptic equation:
\begin{equation*}
	\label{sum}
	\left\{
	\begin{array}{ll}
		-\Delta u(x)=f(u), & \textrm{in} \,\, \Omega, \\
		u(x)=0,& \textrm{on} \, \, \partial \Omega.
	\end{array}
	\right.
\end{equation*}
increases monotonically with respect to $x_n$ in the domain $\Omega$.
Subsequently, Berstycki, Caffarelli, and Nirenberg \cite{C} extended the analysis under certain assumptions on $f$ and observed that the solution need not be bounded.
Dipierro \cite{Dip} generalized the monotonicity results of \cite{Est1982, C} to positive bounded non-decaying solutions for fractional elliptic equations in unbounded domains using a comprehensive version of the sliding method.
In a recent study by Chen \cite{ChenJFA2021}, the author explored bounded solutions of nonlinear equations involving the fractional $p$-Laplacian:
\begin{equation}
	\label{pro:fpL}
	\left\{
	\begin{array}{ll}
		(-\Delta)_p^s u(x)=f(u) & \textrm{in} \ \Omega, \\
		u(x)=0,& \textrm{on}\ \mathbb{R}^n\backslash \Omega
	\end{array}
	\right.
\end{equation}
Here, $\Omega$ represents an epigraph. By estimating the singular integral defining $(-\Delta)_p^s$ along a sequence of auxiliary functions at their maximum points, Chen found that the positive bounded solution of \eqref{pro:fpL} strictly increases with respect to $x_n$ in $\Omega$.
For further research on this topic, interested readers can refer to \cite{MZ}\cite{Wu2021} \cite{Peng2023}\cite{Dai} \cite{CLL22} and the references cited therein.

The logarithmic operator \eqref{def:logl} can be regarded as the first-order derivative of the fractional Laplacian, as shown in greater detail in \cite{CLL1}:
$$
(-\Delta)^s u(x)=u(x)+sL_\Delta u(x)+o(s),\quad \textrm{as}\ s\to 0^+,
$$
for $u\in C_c^2(\mathbb{R}^n)$. 
Furthermore, $L_\Delta$ has a logarithmic symbol $\mathcal{F}(L_\Delta u)(\xi)=(2\ln |\xi|)\hat{u}(\xi)$, $\forall \xi\in \mathbb{R}^n $ \cite{CLL1}, where $\mathcal{F}$ and $\hat{\cdot}$ denote the Fourier transform.
Different from the fractional Laplacian, the order of singular kernel in logarithmic Laplacian is $-n$, resulting in a lack of integrability both locally and at infinity.

Recently, there has been considerable research on topics related to the logarithmic Laplacian, including investigations into eigenvalue estimates \cite{c10}, log-Sobolev inequality \cite{c19}, semilinear problems \cite{c13, c22}, and the Cauchy problem \cite{chy2023}.
In \cite{Liu2018}, we extended the direct method of moving planes to derive symmetric properties of positive solutions for logarithmic Laplacian equations. Additionally, we investigated logarithmic Laplacian equations on unbounded domains, establishing the monotonic behavior of the solutions \cite{Liu2023}.

Inspired by the aforementioned work, it is instinctive to investigate the properties of equations involving the logarithmic Laplacian \eqref{sec:six}. When addressing issues related to the logarithmic Laplacian on an unbounded domain, a primary challenge arises from the non-integrability of the kernel at infinity. Fortunately, when applying the method of moving planes to handle problems on a coercive epigraph, we will encounter the need to work with an antisymmetric function confined to a bounded domain.

The principal results in this paper are as follows:
\begin{theorem} \label{result:min}
Let $\Omega=\{x\in \mathbb{R}^n \,|\, x_n>\varphi(x')\}$ be a Lipschitz coercive epigraph and $u \in L_0\cap C_{loc}^{1,1}(\Omega)\cap C(\bar\Omega)$ be a positive bounded solution of equation \eqref{sec:six}.
Let $I=(l,+\infty).$

Assume
\begin{itemize}
\item[(i)] $a(t)\in C(I)$ and $a(t)$ is nondecreasing in $I$;
\item[(ii)]  $a(t)>0$ for some $t\in I$ and 
\begin{equation}
	\label{ass:lim}
	\lim\limits_{h\to 0}\frac{a(l+h)}{-\ln h}\leq0;
\end{equation}
\item[$(iii)$] $f(\cdot)$ is locally Lipschitz continuous and nondecreasing in $(0, +\infty)$. Moreover, $f(u)>0$ in $(0,+\infty).$ 
\end{itemize}

Then $u$ must be  monotone increasing in $x_n$ direction in $\Omega.$
\end{theorem}

\begin{theorem}\label{result:sum}
Besides the conditions in Theorem \ref{result:min}, further assume that
$$(iv) \quad a(t)\to+\infty,\, \, \textrm{as}\, t\to+\infty.$$
Then equation \eqref{sec:six} possesses no positive bounded solution in $L_0\cap C_{loc}^{1,1}(\Omega)\cap C(\bar\Omega)$.
\end{theorem}

The remaining sections of this paper are structured as follows. In Section \ref{sec:lep}, we utilize the direct method of moving planes to demonstrate the monotonicity of solutions along the $x_n$-direction, as presented in Theorem \ref{result:min}. In Section \ref{sec:leq}, we focus on proving Theorem \ref{result:sum}, establishing the non-existence of positive solutions. Throughout the paper, we will employ the symbol $C$ to represent a constant, which may vary in value from one line to another.

\section{Monotonocity of solutions}\label{sec:lep}

In this section, we will establish the proof of Theorem \ref{result:min} using the direct method of moving planes. To facilitate our analysis for the remainder of the paper, we introduce essential notations and terminologies.

For each $\lambda\in (l, +\infty)$, we write $x=(x',x_n)$ with $x'=(x_1,x_2,\cdot\cdot\cdot,x_{n-1})\in \mathbb{R}^{n-1}$ and define $H_{\lambda}:=\{x\in \Omega \,|\, l<x_n<\lambda\},$ $\Sigma_{\lambda}=\{x\in \mathbb{R}^n~|x_n<\lambda\},$  $T_{\lambda}:=\{x\in \mathbb{R}^n \,|\, x_n=\lambda\}$.  For each
point $x=(x', x_n)\in  \mathbb{R}^n$, let $x^{\lambda} =(x',2\lambda-x_n)$ be the reflected point with respect to the hyperplane $T_{\lambda}.$ Define the reflected functions by $u_{\lambda}(x) = u(x^{\lambda})$ and introduce
function
\begin{equation*}
	\label{def:w} 
	w_{\lambda}(x)=u_{\lambda}(x)-u(x).
\end{equation*}

For all $u\in C_{loc}^{1,1}(\mathbb{R}^n)\cap L_0$, one can compute directly
\begin{align*}
	((-\Delta)^L u_{\lambda})(x)&=C_nP.V.\int_{\mathbb{R}^n}\frac{{ u_{\lambda}(x)1_{B_{1}(x)}}(y)- u_{\lambda}(y)}{|x-y|^n}dy\\		
	&=C_nP.V.\int_{\mathbb{R}^n}\frac{ u(x^{\lambda})1_{B_{1}(x)}(y)- u_{\lambda}(y)}{|x-y|^n}dy\\
	&=C_nP.V.\int_{\mathbb{R}^n}\frac{ u(x^{\lambda})1_{B_{1}(x)}(y^{\lambda})- u(y)}{|x-y^{\lambda}|^n}dy\\
	&=C_nP.V.\int_{\mathbb{R}^n}\frac{{ u(x^{\lambda})1_{B_{1}(x^{\lambda})}}(y)- u(y)}{|x^{\lambda}-y|^n}dy\\
	&=((-\Delta)^L u)(x^{\lambda}),
\end{align*}
where we have used that $|x^{\lambda}-y|=|x-y^{\lambda}|,$ and the reflected domain of $\mathbb{R}^n$ is still $\mathbb{R}^n.$
Thus,
\begin{equation}
	\label{eq:ulam}
	((-\Delta)^L u_{\lambda})(x)=((-\Delta)^L u)(x^{\lambda}), \quad 
	(L_\Delta u_{\lambda})(x)=(L_\Delta  u)(x^{\lambda}) .
\end{equation}

It follows from \eqref{eq:ulam} and assumption $(i), (iii)$ that for all $x\in H_\lambda$, 
\begin{equation}\label{sec:sev}
	\begin{aligned}
		L_\Delta w_{\lambda}(x)&=L_\Delta  u_{\lambda}(x)-L_\Delta u(x)\\
		&=a(2\lambda-x_n)f(u_{\lambda}(x))-a(x_n)f(u(x))\\
		&=\big(a(2\lambda-x_n)-a(x_n)\big)f(u_{\lambda}(x))+a(x_n)\big(f(u_{\lambda}(x))-f(u(x))\big)\\
		&\geq a(x_n)M(\lambda,x)w_{\lambda}(x),
	\end{aligned}
\end{equation}
where
\begin{equation*}\label{2.2}
	M(\lambda,x) = \frac{f(u_{\lambda}(x))-f(u(x))}{u_{\lambda}(x)-u(x)}.
\end{equation*}
When $u(x)$ is bounded and $f(\cdot)$ is locally Lipschitz continuous and nondecreasing in $(0,+\infty)$, we have
\begin{equation}\label{pro:M}
	M(\lambda,x) \ \textrm{is bounded and nonnegative in}\ H_\lambda.
\end{equation}

Let $\widetilde \Omega=\{x^\lambda| x\in \Omega\}$ be the reflected domain of $\Omega$ with respect to the hyperplane $T_\lambda$. 
Denote 
$$A_{\lambda}=\widetilde \Omega\backslash\Omega,\,\,\,
D_{\lambda}=\Sigma_{\lambda}\backslash \widetilde \Omega.$$
Clearly, $\Sigma_{\lambda}=\overline{A_{\lambda}}\cup H_\lambda\cup D_{\lambda}$.

We now present the following maximum principle and boundary estimate lemma, crucial tools that will be utilized throughout this paper.

\begin{lemma}[Strong maximum principle for antisymmetric functions]\label{lem:smp}
	Let $\Omega=\{x\in \mathbb{R}^n \,|\, x_n>\varphi(x')\}$ be a Lipschitz coercive epigraph. Given $l<\lambda<+\infty$, let $w_{\lambda}(x)\in L_0\cap C_{loc}^{1,1}(\Omega)\cap C(\overline{\Sigma_\lambda})$ satisfy
	\begin{equation}
		\label{sec}
		\left\{
		\begin{array}{lc}
			L_\Delta w_{\lambda}(x)\geq a(x_n)M(\lambda,x)w_{\lambda}(x), & x\in H_\lambda, \\
			w_{\lambda}(x)=-w(x),& x\in \Sigma_{\lambda},\\
			w_{\lambda}(x) > 0, & x\in A_\lambda,\\
			w_{\lambda}(x) = 0, & x\in D_\lambda.\\
		\end{array}
		\right.
	\end{equation}
	Suppose $w_{\lambda}(x)\geq 0$ in $H_{\lambda}$, then $w_{\lambda}(x) > 0$, for all $x\in H_{\lambda}$.
\end{lemma}
\begin{proof}
	Assume for contradiction that there is some $x^0\in H_{\lambda}$
	such that
	$w_{\lambda} (x^0)  = 0$.
	It follows from \eqref{sec}, that
	\begin{equation}\label{sec:ten}
		L_\Delta w_{\lambda}(x^0)\geq a(x_n^0)M(\lambda,x^0)w_{\lambda}(x^0)= 0.
	\end{equation}
	
	On the other hand, we have
	\begin{equation}\label{est:delw}
		\begin{aligned}
			L_\Delta w_{\lambda}(x^0)&= (-\Delta)^{L}w_{\lambda}(x^0)\\
			&=C_nP.V.\int_{\mathbb{R}^n}\frac{{w_{\lambda}(x^0)1_{B_{1}(x^0)}}(y)-w_{\lambda}(y)}{|x^0-y|^n}dy\\
			&=C_nP.V.\int_{\mathbb{R}^n}\frac{-w_{\lambda}(y)}{|x^0-y|^n}dy\\
			&=C_n(I+II+III),
		\end{aligned}
	\end{equation}
	in which
	$$
	I=P.V.\int_{H_{\lambda}\cup \widetilde{H}_{\lambda}}\frac{-w_{\lambda}(y)}{|x^0-y|^n}dy,\ 
	II=\int_{A_{\lambda}\cup\widetilde{A}_{\lambda}}\frac{-w_{\lambda}(y)}{|x^0-y|^n}dy,\ 
	III=\int_{D_{\lambda}\cup\widetilde{D}_{\lambda}}\frac{-w_{\lambda}(y)}{|x^0-y|^n}dy.
	$$
	By a straight computation, we get
	\begin{equation*}
		\begin{aligned}
			I 
			&=P.V.\int_{H_{\lambda}}\frac{-w_{\lambda}(y)}{|x^0-y|^n}dy+P.V.\int_{\widetilde{H}_{\lambda}}\frac{-w_{\lambda}(y)}{|x^0-y|^n}dy\\
			&
			=P.V.\int_{H_{\lambda}}\frac{-w_{\lambda}(y)}{|x^0-y|^n}dy+P.V.\int_{{H}_{\lambda}}\frac{-w_{\lambda}(y^{\lambda})}{|x^0-y^{\lambda}|^n}dy\\
			&=P.V.\int_{H_{\lambda}}\bigg(\frac{1}{|x^0-y^{\lambda}|^n}-\frac{1}{|x^0-y|^n}\bigg)w_{\lambda}(y)dy.
		\end{aligned}
	\end{equation*}
	Taking $|x^0-y^{\lambda}|>|x^0-y|$, $\forall y\in \Sigma_\lambda$ and $w_{\lambda}(y)\geq0$, $\forall y\in H_{\lambda}$ into consideration, we derive 
	\begin{equation}
		\label{est:I}
		I\leq 0.
	\end{equation}
	
	Similarly, since  $w_{\lambda}(y)>0$, $\forall y\in A_{\lambda}$,
	\begin{equation}
		\label{est:II}
		II=\int_{A_{\lambda}}\bigg(\frac{1}{|x^0-y^{\lambda}|^n}-\frac{1}{|x^0-y|^n}\bigg)w_{\lambda}(y)dy<0.
	\end{equation}	
	Also, by using the fact that $w_{\lambda}(x) = 0$, $\forall x\in D_\lambda$, there holds	
	\begin{equation}
		\label{est:III}
		III=\int_{D_{\lambda}}\bigg(\frac{1}{|x^0-y^{\lambda}|^n}-\frac{1}{|x^0-y|^n}\bigg)w_{\lambda}(y)dy=0.
	\end{equation}
	
	Consequently, putting \eqref{est:I} \eqref{est:II} \eqref{est:III} into \eqref{est:delw} we derive
	\begin{equation*} 
	L_\Delta w_{\lambda}(x^0)<0.
	\end{equation*}
	which contradicts with \eqref{sec:ten}.
\end{proof}

\begin{lemma} [A boundary estimate  for antisymmetric functions]\label{lem:two}
For some fixed $\lambda_0>l,$ assume $w_{\lambda_0}(x)>0,$ for $x\in H_{\lambda_0}.$ Suppose there are $\lambda_k\searrow\lambda_0,$ and $x^k\in H_{\lambda_k},$ such that 
\begin{equation*}
	w_{\lambda_k}(x^k)=\min\limits_{x\in H_{\lambda_k}}w_{\lambda_k}(x)<0,\, \textrm{and} \, x^k\to x^0\in\partial\Sigma_{\lambda_0}, \textrm{as}\ k\to \infty.
\end{equation*}
Let $\delta_k=dist(x^k,\partial\Sigma_{\lambda_k})\equiv|\lambda_k-x_n^k|.$
Then \begin{equation*}\label{con:delw}
	\overline{\lim\limits_{\delta_k\to 0}}\frac{(-\Delta)^Lw_{\lambda_k}(x^k)}{\delta_k}<0,
	\quad
	\overline{\lim\limits_{\delta_k\to 0}}\frac{L_\Delta w_{\lambda_k}(x^k)}{\delta_k}<0,.
\end{equation*}
\end{lemma}
\begin{proof}
By a similar computation as in \eqref{est:delw}, we have 
\begin{equation}\label{eq:bdest}
	\begin{aligned}
		(-\Delta)^{L}w_{\lambda_k}(x^k)&=C_nP.V.\int_{\mathbb{R}^n}\frac{{w_{\lambda_k}(x^k)1_{B_{1}(x^k)}}(y)-w_{\lambda_k}(y)}{|x^k-y|^n}dy
		=C_n(I+II+III),
	\end{aligned}
\end{equation}
where 
\begin{eqnarray*}
I&=&P.V.\int_{H_{\lambda_k}\cup \widetilde{H}_{\lambda_k}}\frac{{w_{\lambda_k}(x^k)1_{B_{1}(x^k)}}(y)-w_{\lambda_k}(y)}{|x^k-y|^n}dy,\\
II&=&\int_{A_{\lambda_k}\cup\widetilde{A}_{\lambda_k}}\frac{{w_{\lambda_k}(x^k)1_{B_{1}(x^k)}}(y)-w_{\lambda_k}(y)}{|x^k-y|^n}dy,
\\
III&=&\int_{D_{\lambda_k}\cup \widetilde{D}_{\lambda_k}}\frac{{w_{\lambda_k}(x^k)1_{B_{1}(x^k)}}(y)-w_{\lambda_k}(y)}{|x^k-y|^n}dy. 
\end{eqnarray*}

We recall that $H_{\lambda_k}:=\{x\in \Omega \,|\, l<x_n<\lambda_k\}$, $A_{\lambda_k}=\widetilde \Omega\backslash\Omega$,
$D_{\lambda_k}=\Sigma_{\lambda_k}\backslash \widetilde \Omega$, 
and denote 
\begin{equation*} 
	\begin{array}{rl}
		\widetilde{x}^k=(x^k)^{\lambda_k},
		&A=\big(B_1(x^k)\cap H_{\lambda_k}\big)\backslash B_1(\widetilde{x}^k),\\
		B=B_1(\widetilde{x}^k)\cap H_{\lambda_k},
		&C=H_{\lambda_k}\backslash B_1(x^k),\\
		D_1= B_1(\widetilde{x}^k)\cap D_{\lambda_k},
		&D_2=\big(B_1(x^k)\cap D_{\lambda_k}\big)\backslash B_1(\widetilde{x}^k)\\
		E_1=B_1(x^k)\cap A_{\lambda_k}\cap B_1(\widetilde{x}^k),
		&E_2=\big(B_1(x^k)\cap A_{\lambda_k}\big)\backslash B_1(\widetilde{x}^k).
	\end{array}
\end{equation*}

Clearly,	$$E_1\cup E_2=B_1({x^k})\cap A_{\lambda_k},
A\cup B=B_1(x^k)\cap H_{\lambda_k},
D_1\cup D_2=B_1(x^k)\cap D_{\lambda_k}.$$

Since $\lambda_k\searrow\lambda_0,$ and $x^k\to x^0\in T_{\lambda_0}$, we know that for each sufficiently small $h>0$, there is $K>0$ such that when $k>K$, 
\begin{equation}
	\label{con:slk}
	0<\lambda_k-\lambda_0<\frac{h}{2}\quad  \textrm{and}\quad |x^k-x^0|<\frac{h}{2}.
\end{equation}

Thus, the estimate of \eqref{eq:bdest} can be divided into three cases.

\noindent\textbf{Case 1.} $\lambda_0-l<1.$ 

Choose $0<h<1-\lambda_0+l$, there exists $K>0,$ such that for $k>K$, $\lambda_k-l<1$ and $\widetilde{x}^k_n-l<1$. Thus, neither $E_1$ nor $E_2$ is empty. Additionally, notice that if $A=\emptyset$, then $C=\emptyset$, and if $A\neq\emptyset$, then $C\neq\emptyset$. Also, for sufficiently large $k$, $D_1$ and $D_2$ can only be empty at the same time or neither. Therefore, there are three possible sub-cases.

\noindent\textit{Case 1.1:} $A, C,D_1,D_2$ are all non-empty. (See Figure \ref{fig:case1.1}.)

\begin{figure}[htb]
	\begin{center}
		\includegraphics[width=0.5\textwidth]{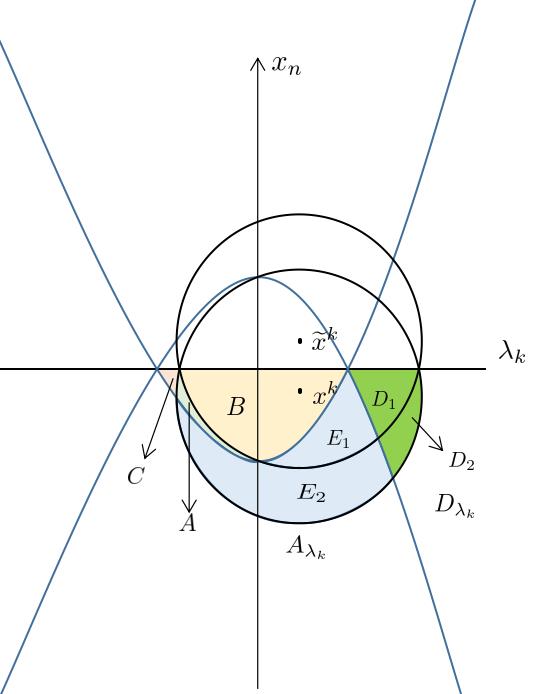}
		\caption{Case 1.1}
		\label{fig:case1.1}
	\end{center}
\end{figure}

A direct calculation shows
\begin{eqnarray}\label{case1-1:I}\\
	\nonumber
	I&=&P.V.\int_{H_{\lambda_k}\cup \widetilde{H}_{\lambda_k}}\frac{{w_{\lambda_k}(x^k)1_{B_{1}(x^k)}}(y)-w_{\lambda_k}(y)}{|x^k-y|^n}dy\\
	&=&P.V.\int_{H_{\lambda_k}\cap B_1(x^k)}\frac{w_{\lambda_k}(x^k)-w_{\lambda_k}(y)}{|x^k-y|^n}dy+P.V.\int_{H_{\lambda_k}\backslash B_1(x^k)}\frac{-w_{\lambda_k}(y)}{|x^k-y|^n}dy\nonumber\\
	&\,&+P.V.\int_{H_{\lambda_k}\cap B_1(\widetilde{x}^k)}\frac{w_{\lambda_k}(x^k)+w_{\lambda_k}(y)}{|x^k-y^{\lambda_k}|^n}dy  
	+P.V.\int_{H_{\lambda_k}\backslash B_1(\widetilde{x}^k)}\frac{w_{\lambda_k}(y)}{|x^k-y^{\lambda_k}|^n}dy\nonumber\\
	&\leq &P.V.\int_{A\cup B}\frac{w_{\lambda_k}(x^k)-w_{\lambda_k}(y)}{|x^k-y^{\lambda_k}|^n}dy+P.V.\int_{C}\frac{-w_{\lambda_k}(y)}{|x^k-y|^n}dy\nonumber\\
	& &\nonumber+P.V.\int_{{B}}\frac{w_{\lambda_k}(x^k)+w_{\lambda_k}(y)}{|x^k-y^{\lambda_k}|^n}dy+P.V.\int_{{A}\cup{C}}\frac{w_{\lambda_k}(y)}{|x^k-y^{\lambda_k}|^n}dy\\
	&=&P.V.\int_{A}\frac{w_{\lambda_k}(x^k)}{|x^k-y^{\lambda_k}|^n}dy+P.V.\int_{B}\frac{2w_{\lambda_k}(x^k)}{|x^k-y^{\lambda_k}|^n}dy+\int_{C}\bigg(\frac{1}{|x^k-y^{\lambda_k}|^n}-\frac{1}{|x^k-y|^n}\bigg)w_{\lambda_k}(y)dy \nonumber\\
	&<&P.V.\int_{A\cup B}\frac{w_{\lambda_k}(x^k)}{|x^k-y^{\lambda_k}|^n}dy+\int_{C}\bigg(\frac{1}{|x^k-y^{\lambda_k}|^n}-\frac{1}{|x^k-y|^n}\bigg)w_{\lambda_k}(y)dy. \nonumber
\end{eqnarray}

It follows from the fact that $u$ is a positive solution of \eqref{sec:six} and hence $w_{\lambda_k}(y)\geq 0$ in $A_{\lambda_k}$, we calculate
\begin{eqnarray}\label{case1-1:II}\\
	\nonumber
	II&=& \int_{A_{\lambda_k}\cup \widetilde{A}_{\lambda_k}}\frac{{w_{\lambda_k}(x^k)1_{B_{1}(x^k)}}(y)-w_{\lambda_k}(y)}{|x^k-y|^n}dy \\
	&=& \int_{A_{\lambda_k}}\frac{{w_{\lambda_k}(x^k)1_{B_{1}(x^k)}}(y)-w_{\lambda_k}(y)}{|x^k-y|^n}dy+\int_{ {A_{\lambda_k}}}\frac{{w_{\lambda_k}(x^k)1_{B_{1}(x^k)}}(y^{\lambda_k})-w_{\lambda_k}(y^{\lambda_k})}{|x^k-y^{\lambda_k}|^n}dy \nonumber \\
	&=& \int_{A_{\lambda_k}\cap B_1(x^k)}\frac{w_{\lambda_k}(x^k)-w_{\lambda_k}(y)}{|x^k-y|^n}dy+\int_{A_{\lambda_k}\backslash B_1(x^k)}\frac{-w_{\lambda_k}(y)}{|x^k-y|^n}dy\nonumber\\
	&\,&+\int_{A_{\lambda_k}\cap B_1(\widetilde{x}^k)}\frac{w_{\lambda_k}(x^k)+w_{\lambda_k}(y)}{|x^k-y^{\lambda_k}|^n}dy  
	+\int_{A_{\lambda_k}\backslash B_1(\widetilde{x}^k)}\frac{w_{\lambda_k}(y)}{|x^k-y^{\lambda_k}|^n}dy\nonumber\nonumber\\
	&=& \int_{E_1\cup E_2}\frac{w_{\lambda_k}(x^k)-w_{\lambda_k}(y)}{|x^k-y|^n}dy+\int_{A_{\lambda_k}\backslash B_1(x^k)}\frac{-w_{\lambda_k}(y)}{|x^k-y|^n}dy\nonumber\\
	&\,&+\int_{E_1}\frac{w_{\lambda_k}(x^k)+w_{\lambda_k}(y)}{|x^k-y^{\lambda_k}|^n}dy  
	+\int_{E_2\cup(A_{\lambda_k}\backslash B_1({x}^k))}\frac{w_{\lambda_k}(y)}{|x^k-y^{\lambda_k}|^n}dy\nonumber\nonumber\\
	&\leq& \int_{E_1}\frac{w_{\lambda_k}(x^k)}{|x^k-y^{\lambda_k}|^n}dy+\int_{E_2}\frac{w_{\lambda_k}(x^k)}{|x^k-y^{\lambda_k}|^n}dy+\int_{A_{\lambda_k}\backslash B_1(x^k)}\left(\frac{1}{|x^k-y^{\lambda_k}|^n}-\frac{1}{|x^k-y|^n}\right)w_{\lambda_k}(y)dy\nonumber\\
	&\leq& \int_{E_1\cup E_2}\frac{w_{\lambda_k}(x^k)}{|x^k-y^{\lambda_k}|^n}dy\nonumber \\
	&<& 0\nonumber.
\end{eqnarray}
By using the fact that $w_{\lambda_k}(y)=0$ in $D_{\lambda_k}$, we derive 
\begin{eqnarray}\label{case1-1:III}
	III&=&\int_{D_{\lambda_k}\cup \widetilde{D}_{\lambda_k}}\frac{{w_{\lambda_k}(x^k)1_{B_{1}(x^k)}}(y)-w_{\lambda_k}(y)}{|x^k-y|^n}dy \\
	&=&\int_{D_{\lambda_k}}\frac{{w_{\lambda_k}(x^k)1_{B_{1}(x^k)}}(y)-w_{\lambda_k}(y)}{|x^k-y|^n}dy+\int_{ {D_{\lambda_k}}}\frac{{w_{\lambda_k}(x^k)1_{B_{1}(x^k)}}(y^{\lambda_k})-w_{\lambda_k}(y^{\lambda_k})}{|x^k-y^{\lambda_k}|^n}dy \nonumber \\
	&=&\int_{D_{\lambda_k}\cap B_1(x^k)}\frac{w_{\lambda_k}(x^k)-w_{\lambda_k}(y)}{|x^k-y|^n}dy+\int_{D_{\lambda_k}\backslash B_1(x^k)}\frac{-w_{\lambda_k}(y)}{|x^k-y|^n}dy\nonumber\\
	&\,&+\int_{D_{\lambda_k}\cap B_1(\widetilde{x}^k)}\frac{w_{\lambda_k}(x^k)+w_{\lambda_k}(y)}{|x^k-y^{\lambda_k}|^n}dy  
	+\int_{D_{\lambda_k}\backslash B_1(\widetilde{x}^k)}\frac{w_{\lambda_k}(y)}{|x^k-y^{\lambda_k}|^n}dy\nonumber\nonumber\\
	&\leq&\int_{D_1\cup D_2}\frac{w_{\lambda_k}(x^k)-w_{\lambda_k}(y)}{|x^k-y^{\lambda_k}|^n}dy+P.V.\int_{D_1}\frac{w_{\lambda_k}(x^k)+w_{\lambda_k}(y)}{|x^k-y^{\lambda_k}|^n}dy\nonumber\\
	&<&\int_{D_1\cup D_2}\frac{w_{\lambda_k}(x^k)}{|x^k-y^{\lambda_k}|^n}dy.\nonumber\\
	&<& 0.\nonumber
\end{eqnarray}
Therefore, we obtain $$(-\Delta)^Lw_{\lambda_k}(x^k)\leq C_nP.V.\int_{M_1}\frac{w_{\lambda_k}(x^k)}{|x^k-y^{\lambda_k}|^n}dy+\int_C\left(\frac{1}{|x^k-y^{\lambda_k}|^n}-\frac{1}{|x^k-y|^n}\right)w_{\lambda_k}(y)dy.$$
where $M_1=A\cup B\cup E_1\cup E_2\cup D_1\cup D_2=\Sigma_{\lambda_k}\cap B_{1}(x^k).$

\noindent\textit{Case 1.2:} $A, C, D_1, D_2$ are all empty.  (See Figure \ref{fig:case1.2}.)

\begin{figure}[htb]
	\begin{center}
		\includegraphics[width=0.5\textwidth]{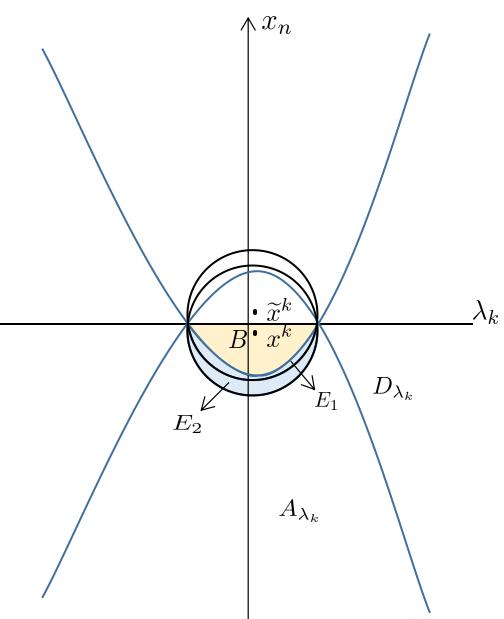}
		\caption{Case 1.2}
		\label{fig:case1.2}
	\end{center}
\end{figure}

By using a similar computation as in \eqref{case1-1:I}, we deduce
\begin{eqnarray*}\label{case1-1:I0}
	I
	&<&P.V.\int_{B}\frac{w_{\lambda_k}(x^k)}{|x^k-y^{\lambda_k}|^n}dy.
\end{eqnarray*}
In this case, as in \textit{Case 1.1}, the calculations for II are identical, we also have \eqref{case1-1:II}.

As to term $III$, since $w_{\lambda_k}\equiv0$ in $D_{\lambda_k}$ and $D_1=D_2=\emptyset$, we know that 
$D_{\lambda_k}\cap B_1(x^k)=D_{\lambda_k}\cap B_1(\widetilde{x}^k)=\emptyset$. Hence,
\begin{eqnarray}
	III &=&\int_{D_{\lambda_k}\cup \widetilde{D}_{\lambda_k}}\frac{{w_{\lambda_k}(x^k)1_{B_{1}(x^k)}}(y)-w_{\lambda_k}(y)}{|x^k-y|^n}dy\nonumber \\
	&=&\int_{D_{\lambda_k}}\frac{{w_{\lambda_k}(x^k)1_{B_{1}(x^k)}}(y)-w_{\lambda_k}(y)}{|x^k-y|^n}dy+\int_{ {D_{\lambda_k}}}\frac{{w_{\lambda_k}(x^k)1_{B_{1}(x^k)}}(y^{\lambda_k})-w_{\lambda_k}(y^{\lambda_k})}{|x^k-y^{\lambda_k}|^n}dy \nonumber \\
	&=& \int_{D_{\lambda_k}\cap B_1(x^k)}\frac{w_{\lambda_k}(x^k)-w_{\lambda_k}(y)}{|x^k-y|^n}dy+\int_{D_{\lambda_k}\backslash B_1(x^k)}\frac{-w_{\lambda_k}(y)}{|x^k-y|^n}dy\nonumber\\
	&\,&+\int_{D_{\lambda_k}\cap B_1(\widetilde{x}^k)}\frac{w_{\lambda_k}(x^k)+w_{\lambda_k}(y)}{|x^k-y^{\lambda_k}|^n}dy  
	+\int_{D_{\lambda_k}\backslash B_1(\widetilde{x}^k)}\frac{w_{\lambda_k}(y)}{|x^k-y^{\lambda_k}|^n}dy\nonumber\nonumber\\
	&=&0.\nonumber
\end{eqnarray} 	
In this case, we find 
$$
(-\Delta)^Lw_{\lambda_k}(x^k)\leq C_nP.V.\int_{M_2}\frac{w_{\lambda_k}(x^k)}{|x^k-y^{\lambda_k}|^n}dy,
$$
in which $M_2= B\cup E_1\cup E_2=\Sigma_{\lambda_k}\cap B_{1}(x^k).$

\noindent\textit{Case 1.3:} $A,C=\emptyset, D_1, D_2\neq\emptyset.$  (See Figure \ref{fig:case1.3}.)

\begin{figure}[htb]
	\begin{center}
		\includegraphics[width=0.5\textwidth]{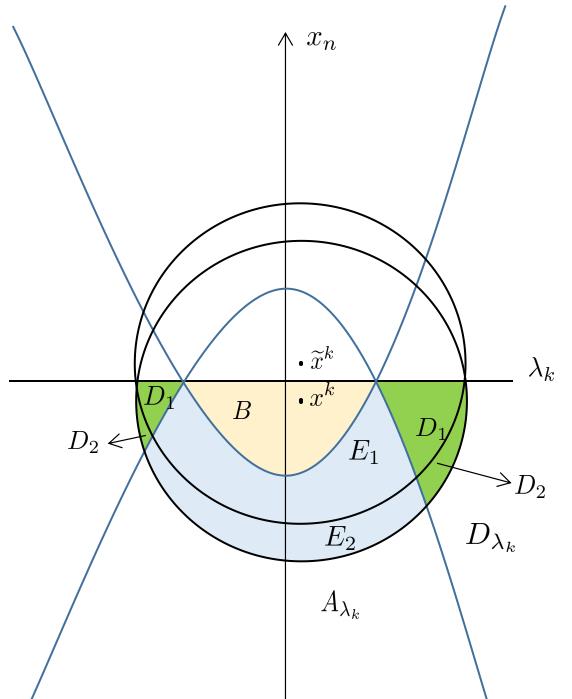}
		\caption{Case 1.3}
		\label{fig:case1.3}
	\end{center}
\end{figure}

The calculation of term	$I$  is the same as in \textit{Case 1.2}.
The calculation of $II$ and $III$ are the same as in \textit{Case 1.1}.

In this case, we obtain 
\begin{equation*}
	\label{con:case1.3}
	(-\Delta)^Lw_{\lambda_k}(x^k)\leq C_nP.V.\int_{M_3}\frac{w_{\lambda_k}(x^k)}{|x^k-y^{\lambda_k}|^n}dy,
\end{equation*}
where $M_3=B\cup E_1\cup E_2\cup D_1\cup D_2=\Sigma_{\lambda_k}\cap B_{1}(x^k).$

\noindent {$\mathbf{Case \, 2}$\label{Case2}} : $\lambda_0-l>1.$ 

By the assumption of $\lambda_k,$ we must have $\lambda_k-l>1.$ In this case, $C$ and $A$ are all non-empty. For sufficiently large $k$, $D_1$, $D_2$ can only be empty at the same time or neither. Moreover, if $E_1\neq \emptyset$, then $E_2\neq \emptyset$.  Therefore, there are four possible sub-cases.

\noindent\textit{Case 2.1:} $D_1, D_2, E_1, E_2$ are all empty.  (See Figure \ref{fig:case2.1}.)

\begin{figure}[htb]
	\begin{center}
		\includegraphics[width=0.5\textwidth]{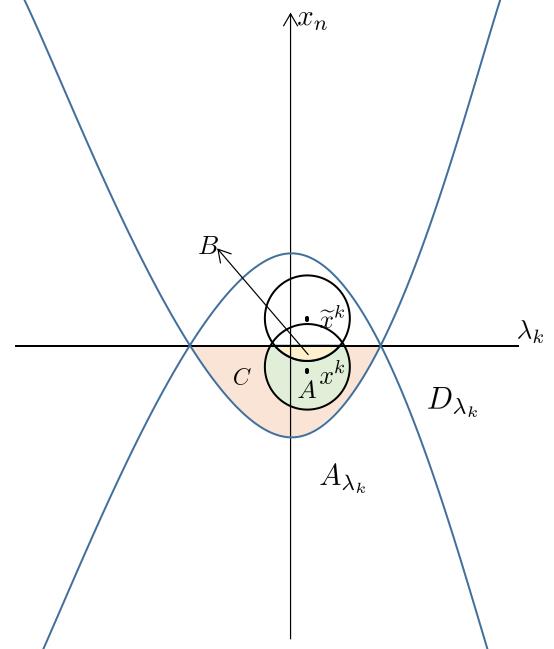}
		\caption{Case 2.1}
		\label{fig:case2.1}
	\end{center}
\end{figure}

By using a similar estimate as in \textit{Case 1.1}, we obtain 
\begin{equation*}
\label{case2-1:I}
I<P.V.\int_{A\cup B}\frac{w_{\lambda_k}(x^k)}{|x^k-y^{\lambda_k}|^n}dy+\int_{C}\bigg(\frac{1}{|x^k-y^{\lambda_k}|^n}-\frac{1}{|x^k-y|^n}\bigg)w_{\lambda_k}(y)dy.
\end{equation*}
Moreover, we know that $III=0$, by a same argument as in  \textit{Case 1.2}.

As to term $II$, since $A_{\lambda_k}\cap B_1(x^k)=A_{\lambda_k}\cap B_1(\widetilde{x}^k)=\emptyset$, $w_{\lambda_k}(x^k)>0$ in $A_{\lambda_k}$ and $|x^k-y|<|x^k-y^{\lambda_k}|$ for all $y\in A_{\lambda_k}$, we calculate
\begin{eqnarray}
	II&=&\int_{A_{\lambda_k}\cup \widetilde{A}_{\lambda_k}}\frac{{w_{\lambda_k}(x^k)1_{B_{1}(x^k)}}(y)-w_{\lambda_k}(y)}{|x^k-y|^n}dy\nonumber \\
	&=&\int_{A_{\lambda_k}\cap B_1(x^k)}\frac{w_{\lambda_k}(x^k)-w_{\lambda_k}(y)}{|x^k-y|^n}dy+\int_{A_{\lambda_k}\backslash B_1(x^k)}\frac{-w_{\lambda_k}(y)}{|x^k-y|^n}dy\nonumber\\
	&\,&+\int_{A_{\lambda_k}\cap B_1(\widetilde{x}^k)}\frac{w_{\lambda_k}(x^k)+w_{\lambda_k}(y)}{|x^k-y^{\lambda_k}|^n}dy  
	+\int_{A_{\lambda_k}\backslash B_1(\widetilde{x}^k)}\frac{w_{\lambda_k}(y)}{|x^k-y^{\lambda_k}|^n}dy\nonumber\nonumber\\
	&=&\int_{A_{\lambda_k}}\frac{-w_{\lambda_k}(y)}{|x^k-y|^n}dy+\int_{A_{\lambda_k}}\frac{w_{\lambda_k}(y)}{|x^k-y^{\lambda_k}|^n}dy\nonumber\\
	&=&\int_{A_{\lambda_k}}\left(\frac{1}{|x^k-y^{\lambda_k}|^n}-\frac{1}{|x^k-y|^n}\right)w_{\lambda_k}(y)dy\nonumber\\
	&<&0.\nonumber 
\end{eqnarray}
So we deduce 
$$
(-\Delta)^Lw_{\lambda_k}(x^k)\leq C_nP.V.\int_{M_4}\frac{w_{\lambda_k}(x^k)}{|x^k-y^{\lambda_k}|^n}dy+\int_{C}\bigg(\frac{1}{|x^k-y^{\lambda_k}|^n}-\frac{1}{|x^k-y|^n}\bigg)w_{\lambda_k}(y)dy.
$$
where $M_4=A\cup B=\Sigma_{\lambda_k}\cap B_{1}(x^k).$

\noindent\textit{Case 2.2:} $D_1, D_2, E_1, E_2$ are all non-empty.  (See Figure \ref{fig:case2.2}.)

\begin{figure}[htb]
	\begin{center}
		\includegraphics[width=0.5\textwidth]{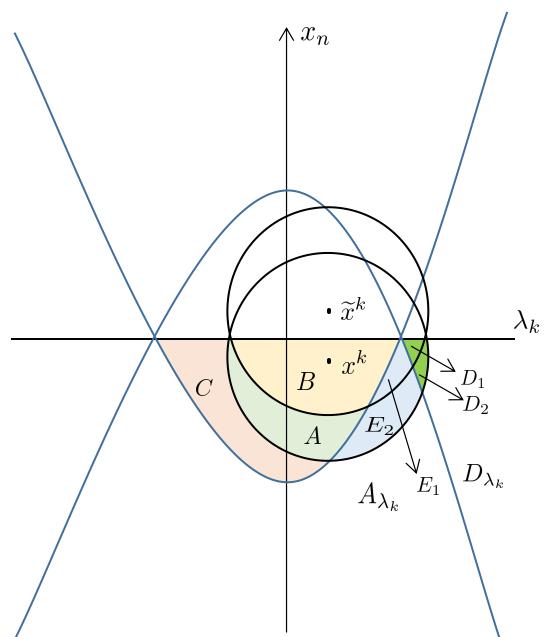}
		\caption{Case 2.2}
		\label{fig:case2.2}
	\end{center}
\end{figure}
This calculation is an exact duplicate of that in \textit{Case 1.1}, which we can also obtain $$(-\Delta)^Lw_{\lambda_k}(x^k)\leq C_nP.V.\int_{M_5}\frac{w_{\lambda_k}(x^k)}{|x^k-y^{\lambda_k}|^n}dy+\int_C\left(\frac{1}{|x^k-y^{\lambda_k}|^n}-\frac{1}{|x^k-y|^n}\right)w_{\lambda_k}(y)dy.$$
where $M_5=A\cup B\cup E_1\cup E_2\cup D_1\cup D_2=\Sigma_{\lambda_k}\cap B_1(x^k).$

\noindent\textit{Case 2.3:} $D_1,D_2=\emptyset, E_1,E_2\neq\emptyset.$  (See Figure \ref{fig:case2.3}.)

\begin{figure}[htb]
	\begin{center}
		\includegraphics[width=0.5\textwidth]{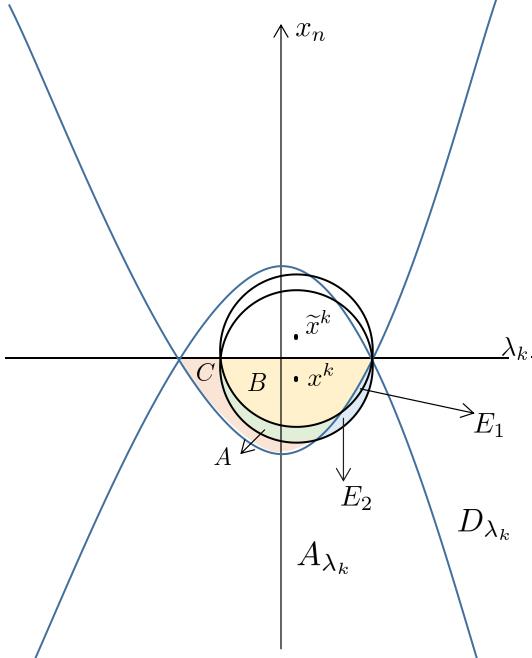}
		\caption{Case 2.3}
		\label{fig:case2.3}
	\end{center}
\end{figure}

The calculations for $I$  and $II$ are identical to those in \textit{Case 1.1} and 
the calculations for $III$  is identical to the one in 
\textit{Case 1.2}. 
Hence we deduce 
$$
(-\Delta)^Lw_{\lambda_k}(x^k)\leq C_nP.V.\int_{M_6}\frac{w_{\lambda_k}(x^k)}{|x^k-y^{\lambda_k}|^n}dy+\int_C\left(\frac{1}{|x^k-y^{\lambda_k}|^n}-\frac{1}{|x^k-y|^n}\right)w_{\lambda_k}(y)dy,
$$
where $M_6=A\cup B\cup E_1\cup E_2=\Sigma_{\lambda_k}\cap B_1(x^k).$

\noindent\textit{Case 2.4:} $D_1=D_2=E_1=\emptyset, E_2\neq\emptyset.$

\begin{figure}[htb]
	\begin{center}
		\includegraphics[width=0.5\textwidth]{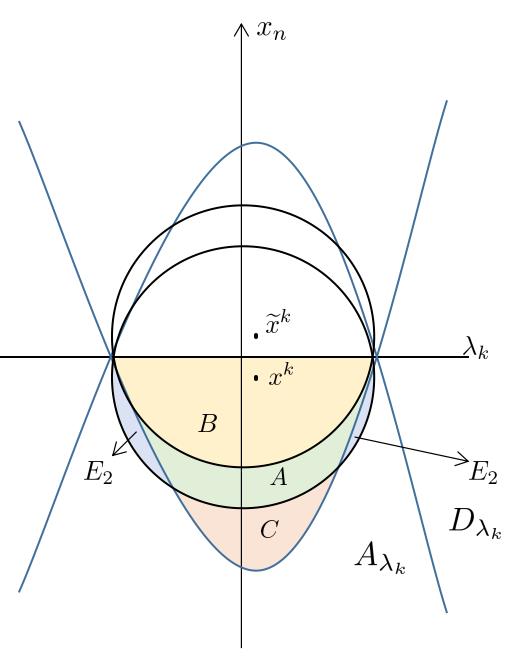}
		\caption{Case 2.4}
		\label{fig:case2.4}
	\end{center}
\end{figure}
The calculation of $I$ is the same as \textit{Case1.1}, and the calculation of $III$ is the same as \textit{Case1.3}.

By using a similar estimate as in \textit{Case1.1},  we obtain
\begin{eqnarray*}
	II=P.V.\int_{A_{\lambda_k}\cup \widetilde{A}_{\lambda_k}}\frac{{w_{\lambda_k}(x^k)1_{B_{1}(x^k)}}(y)-w_{\lambda_k}(y)}{|x^k-y|^n}dy\nonumber 
	\leq P.V.\int_{E_2}\frac{w_{\lambda_k}(x^k)}{|x^k-y^{\lambda_k}|^n}dy.\nonumber 
\end{eqnarray*}
Hence, we arrive at
$$
(-\Delta)^Lw_{\lambda_k}(x^k)\leq C_nP.V.\int_{M_7}\frac{w_{\lambda_k}(x^k)}{|x^k-y^{\lambda_k}|^n}dy+\int_C\big(\frac{1}{|x^k-y^{\lambda_k}|^n}-\frac{1}{|x^k-y|^n}\big)w_{\lambda_k}(y)dy,
$$
where $M_7=A\cup B\cup E_2=\Sigma_{\lambda_k}\cap B_1(x^k).$

\noindent {$\mathbf{Case \, 3}$} : $\lambda_0-l=1.$  

In this case, if $x^k_n\in (\lambda_0, \lambda_k)$, the discussion is analogous to \textit{Case 2}, so we omit it here.
Hence, we only consider  $x^k_n<\lambda_0$, which implies $A, E_1, E_2, D_1, D_2$ are all non-empty.

\noindent\textit{Case 3.1:} $C\neq \emptyset$. (See Figure \ref{fig:case3.1}.)
\begin{figure}[htb]
	\begin{center}
		\includegraphics[width=0.5\textwidth]{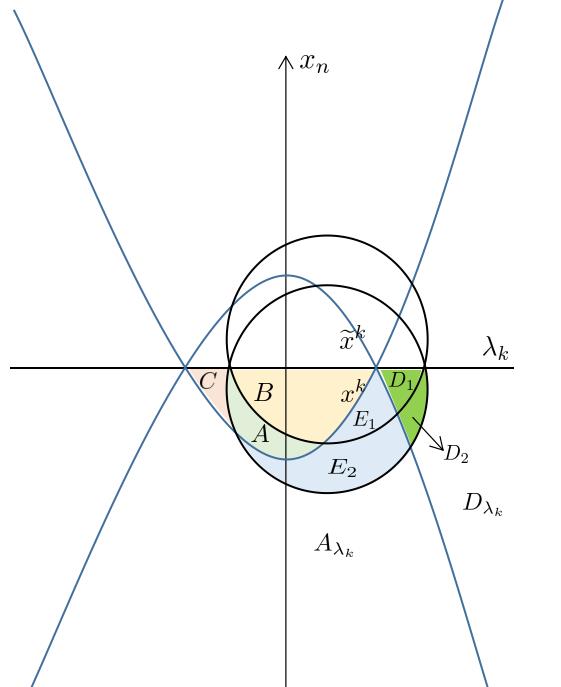}
		\caption{Case 3.1}
		\label{fig:case3.1}
	\end{center}
\end{figure}
The situation is similar to that of \textit{Case 1.1}.

\noindent\textit{Case 3.2:} $C=\emptyset$. (See Figure \ref{fig:case3.2}.)
\begin{figure}[htb]
	\begin{center}
		\includegraphics[width=0.5\textwidth]{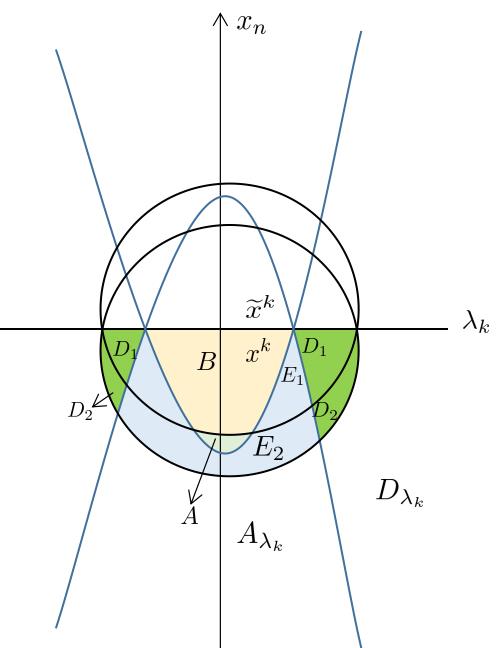}
		\caption{Case 3.2}
		\label{fig:case3.2}
	\end{center}
\end{figure}

The calculations of terms $II$ and $III$ are identical to those in \textit{Case 1.1} and there hold \eqref{case1-1:II} and \eqref{case1-1:III}. In order to estimate term $I$, we use the anti-symmetry property of $w_{\lambda_k}$ and $C=H_{\lambda_k}\backslash B_1(x^k)=\emptyset,$ to derive that
\begin{eqnarray*}\label{tab}
	I&=&P.V.\int_{H_{\lambda_k}\cup \widetilde{H}_{\lambda_k}}\frac{{w_{\lambda_k}(x^k)1_{B_{1}(x^k)}}(y)-w_{\lambda_k}(y)}{|x^k-y|^n}dy\nonumber\\
	&=&P.V.\int_{H_{\lambda_k}\cap B_1(x^k)}\frac{w_{\lambda_k}(x^k)-w_{\lambda_k}(y)}{|x^k-y|^n}dy+P.V.\int_{H_{\lambda_k}\backslash B_1(x^k)}\frac{-w_{\lambda_k}(y)}{|x^k-y|^n}dy\nonumber\\
	&\,&+P.V.\int_{H_{\lambda_k}\cap B_1(\widetilde{x}^k)}\frac{w_{\lambda_k}(x^k)+w_{\lambda_k}(y)}{|x^k-y^{\lambda_k}|^n}dy  
	+P.V.\int_{H_{\lambda_k}\backslash B_1(\widetilde{x}^k)}\frac{w_{\lambda_k}(y)}{|x^k-y^{\lambda_k}|^n}dy\\
	&\leq &P.V.\int_{A\cup B}\frac{w_{\lambda_k}(x^k)-w_{\lambda_k}(y)}{|x^k-y^{\lambda_k}|^n}dy+P.V.\int_{{B}}\frac{w_{\lambda_k}(x^k)+w_{\lambda_k}(y)}{|x^k-y^{\lambda_k}|^n}dy+\int_{{A}}\frac{w_{\lambda_k}(y)}{|x^k-y^{\lambda_k}|^n}dy\nonumber\\
	&= &P.V.\int_{A}\frac{w_{\lambda_k}(x^k)}{|x^k-y^{\lambda_k}|^n}dy+P.V.\int_{{B}}\frac{2w_{\lambda_k}(x^k)}{|x^k-y^{\lambda_k}|^n}dy\nonumber\\
	&<&P.V.\int_{A\cup B}\frac{w_{\lambda_k}(x^k)}{|x^k-y^{\lambda_k}|^n}dy. \nonumber
\end{eqnarray*}
Therefore, we obtain
$$
(-\Delta)^Lw_{\lambda_k}(x^k)\leq C_nP.V.\int_{M_8}\frac{w_{\lambda_k}(x^k)}{|x^k-y^{\lambda_k}|^n}dy,
$$
where $M_8=A\cup B=\Sigma_{\lambda_k}\cap B_1(x^k).$

To conclude, combining all the above scenarios, we have
\begin{equation}\label{est:delw0}
	(-\Delta)^Lw_{\lambda_k}(x^k)\leq  C_n (I_{1k}+I_{2k}),
\end{equation}
in which 
$$
I_{1k}=w_{\lambda_k}(x^k) \int_{\Sigma_{\lambda_k}\cap B_1(x^k)}\frac{1}{|x^k-y^{\lambda_k}|^n}dy,\quad 
I_{2k}=\int_{C}\bigg(\frac{1}{|x^k-y^{\lambda_k}|^n}-\frac{1}{|x^k-y|^n}\bigg)w_{\lambda_k}(y)dy. 
$$

Let $H=\{y~|~h/2<y_n-x_n^k<\frac{1}{4},|y'-(x^k)'|<\frac{1}{8}\}\subset B_1(x^k)\cap\widetilde{\Sigma}_{\lambda_k}$, where $h>0$ defined as in \eqref{con:slk} is sufficiently small and will be chosen later.
Set $s=y_n-x_n^k, \tau=|y'-(x^k)'|$ and $\omega_{n-2}=|B_1(0)|$ in $ R^{n-2}.$ 
We can estimate the integral $\int_{\Sigma_{\lambda_k}\cap B_1(x^k)}\frac{1}{|x^k-y^{\lambda_k}|^n}dy$ as follows:
\begin{equation}\label{est:infi}
	\begin{aligned}
		\int_{\Sigma_{\lambda_k}\cap B_1(x^k)}\frac{1}{|x^k-y^{\lambda_k}|^n}dy&\geq \int_{H}\frac{1}{|x^k-y|^n}dy\\
		&=\int_{h/2}^{\frac{1}{4}}\int_{0}^{\frac{1}{8}}\frac{\omega_{n-2}\tau^{n-2}}{(s^2+\tau^2)^{\frac{n}{2}}}d\tau ds
		=\int_{h/2}^{\frac{1}{4}}\frac{1}{s}\int_{0}^{\frac{1}{8s}}\frac{\omega_{n-2}t^{n-2}}{(1+t^2)^{\frac{n}{2}}}dtds\\
		&
		\geq\int_{h/2}^{\frac{1}{4}}\frac{1}{s}\int_{0}^{\frac{1}{2}}\frac{\omega_{n-2}t^{n-2}}{(1+t^2)^{\frac{n}{2}}}dtds= c\int_{h/2}^{\frac{1}{4}}\frac{1}{s}ds\\
		&
		=c(\ln\frac{1}{4}-\ln h)\to+\infty, \textrm{as}\ h\to 0^+.
	\end{aligned}
\end{equation}
where $c>0$ is a constant.
Now we choose $h>0$ so that $\ln\frac{1}{4}-\ln h>0$, which leads to 
\begin{equation*}
	I_{1k}=w_{\lambda_k}(x^k) P.V.\int_{\Sigma_{\lambda_k}\cap B_1(x^k)}\frac{1}{|x^k-y^{\lambda_k}|^n}dy<0,
\end{equation*}
for sufficiently large $k$. 
Hence 
\begin{equation}
	\label{est:ik1}
	\overline{\lim\limits_{\delta_k\to 0}}\frac{I_{1k}}{\delta_k}\leq 0.
\end{equation}

In order to estimate $I_{2k}$, we apply the mean value theorem on $f(t) = t^{\frac{-n}{2}}$ with $t=|\cdot\cdot\cdot |^2$ to obtain 
$$
\frac{1}{\delta_k}\left(\frac{1}{|x^k-y^{\lambda_k}|^n}-\frac{1}{|x^k-y|^n}\right)=-\frac{2n(\lambda_k-y_n)}{|\eta_k(y)|^{(n+2)}}
\to -\frac{2n(\lambda_0-y_n)}{|\eta_0(y)|^{(n+2)}}<0, \textrm{as}\ k\to \infty,
$$
where $|\eta_k(y)|\in (|x^k-y|,|x^k-y^{\lambda_k}|)$ and $|\eta_0(y)|\in (|x^0-y|,|x^0-y^{\lambda_0}|)$.
Meanwhile, 
$$w_{\lambda_k}(y)\to w_{\lambda_0}(y)>0, \quad \forall y\in C,
$$
which leads to 
\begin{equation}
	\label{est:ik2}\overline{\lim\limits_{\delta_k\to 0}}\frac{I_{2k}}{\delta_k}<0.
\end{equation}

Therefore, putting \eqref{est:ik1} and \eqref{est:ik2} into \eqref{est:delw0}, we obtain
$
	\overline{\lim\limits_{\delta_k\to 0}}\frac{(-\Delta)^Lw_{\lambda_k}(x^k)}{\delta_k}<0.
$
Notice that 
$$
\frac{L_\Delta w_{\lambda_k}(x^k)}{\delta_k}
=\frac{(-\Delta)^Lw_{\lambda_k}(x^k)+\rho_n w_{\lambda_k}(x^k)}{\delta_k}
<\frac{(-\Delta)^Lw_{\lambda_k}(x^k)}{\delta_k},
$$
we get
$$
\overline{\lim\limits_{\delta_k\to 0}}\frac{L_\Delta w_{\lambda_k}(x^k)}{\delta_k}<0.
$$

This completes the proof of Lemma \ref{lem:smp}.
\end{proof}

Lemma \ref{lem:smp} is a key ingredient in applying the direct method of moving planes. 
We now prove that the positive solution of problem \eqref{sec:six} must be strictly monotone increasing along $x_n$ direction.
\begin{proof}[Proof of Theorem \ref{result:min}] 
The conclusion is equivalent to that for all $\lambda>l$, $w_\lambda (x)\geq 0$, $\forall x\in H_\lambda.$ Thanks to Lemma \ref{lem:smp}, it is enough to show that for each $\lambda>l$, there holds 
\begin{equation}
	\label{eq:aim}
	w_\lambda (x)\geq 0,\quad \forall x\in H_\lambda.
\end{equation}

The proof is divided into two steps.

\textbf{Step 1.} We show that for $\lambda>l$ and sufficiently closed to $l$, \eqref{eq:aim} holds.

Suppose otherwise, then there exists a point $x^0\in H_{\lambda}$, such that $w_{\lambda}(x^0)<0$.
Notice that $w_{\lambda}(x)\geq 0$, for all $x\in \Sigma_\lambda \backslash H_\lambda$ and $w_\lambda(x)\in C(\bar \Omega)$. Without loss of generalization, let
$$
w_{\lambda}(x^0)=\min\limits_{x\in \Sigma_{\lambda}}w_{\lambda}(x) < 0.
$$ 

Similarly to the calculation of \textit{Case 1.3} in Lemma \ref{lem:two}, we derive
\begin{equation*}\label{sec:ses}
	\begin{aligned}
	L_\Delta w_{\lambda}(x^0)<	(-\Delta)^{L}w_{\lambda}(x^0)&\leq C_n w_{\lambda}(x^0) P.V.\int_{\Sigma_{\lambda}\cap B_1(x^0)}\frac{1}{|x^0-y^{\lambda}|^n}dy.
	\end{aligned}
\end{equation*}
For $0<\lambda-l<\frac{m}{2}$, in which $m>0$ is sufficiently small, by using a similar argument as in \eqref{est:infi}, we estimate
\begin{equation*}
		\int_{\Sigma_{\lambda}\cap B_1(x^0)}\frac{1}{|x^0-y^{\lambda}|^n}dy
		\geq c(\ln\frac{1}{4}-\ln m)\to+\infty, as~m\to 0^+.
\end{equation*}
where $c>0$ is a constant.
Thus, we have
\begin{equation}\label{sec:eie}
	\begin{aligned}
		L_\Delta w_{\lambda}(x^0)
		&\leq cw_{\lambda}(x^0)(\ln \frac{1}{4}-\ln m)
	\end{aligned}
\end{equation}
Combining \eqref{sec:sev} with \eqref{sec:eie}, we arrive at
\begin{equation*}
	a(x_n^0)M(\lambda,x^0)\geq c(\ln \frac{1}{4}-\ln m).
\end{equation*}
Choose $m>0$ sufficiently small so that $\ln \frac{1}{4}-\ln m>0$, which leads to
\begin{equation*}
	\frac{a(x_n^0)}{\ln \frac{1}{4}-\ln m}M(\lambda,x^0)\geq c.
\end{equation*}
According to assumption (i), \eqref{pro:M} and $a(x_n^0)\leq a(m+l)$ since $x_0\in H_\lambda$, $x_n^0<\lambda<m+l$, we have 
\begin{equation}\label{cont:up}
	\frac{a(l+m)}{\ln \frac{1}{4}-\ln m}M(\lambda,x^0)\geq c>0.
\end{equation}
Inequality \eqref{cont:up} contradicts with assumption \eqref{ass:lim}, when $m$ is sufficiently small.

This completes step 1. 

Step 1 provides a starting point to move the plane. Now we move the hyperplane $T_{\lambda}$ towards upside as long as inequality \eqref{eq:aim} holds.
Define
$$
\lambda_0=\sup\{\lambda\ |\ w_{\mu}(x)\geq 0, x\in H_{\mu}, \forall \mu \in (l,\lambda)\}.
$$

\textbf{Step 2.} We will show that $T_{\lambda}$ can be moved all the way to infinity, i.e. $\lambda_0=+\infty$.

Assume for contradiction that $\lambda_0$ is finite. It follows from the definition of $\lambda_0$ that there exist $\{\lambda_k\}_{k=1}^\infty$ and $\{x^k\}_{k=1}^\infty$ satisfying $\lambda_0<\lambda_{k+1}<\lambda_k$, $\lim_{k\to\infty}\lambda_k= \lambda_0$, $x^k\in H_{\lambda_k}$, such that
\begin{equation}\label{sec:nin}
	w_{\lambda_k}(x^k)=\min\limits_{x\in H_{\lambda_k}}w_{\lambda_k}(x)<0.
\end{equation}

In view of the fact that $H_{\lambda_k}\subset H_{\lambda_1}$, the sequence $\{x^k\}_{k=1}^\infty$ is  bounded. Thus,
there is a sub-sequence of $\{x^k\}$ (still denote it as $\{x^k\}$) converges to some point $x^0\in \overline{H}_{\lambda_0}$. 
Hence, from \eqref{sec:nin}, we have
\begin{equation}
	\label{prop:wlam0-1}
	w_{\lambda_0}(x^0)\leq 0.
\end{equation}

On the other hand, by using the definition of $\lambda_0$ and the continuity, we know that 
$$
w_{\lambda_0}(x) \geq 0,\quad \forall x\in H_{\lambda_0}.
$$
Using the above inequality and Lemma \ref{lem:smp}, we obtain
\begin{equation}
	\label{prop:wlam0-2}
	w_{\lambda_0}(x)>0,\quad \forall x\in H_{\lambda_0}.
\end{equation} 

Combining \eqref{prop:wlam0-1} and \eqref{prop:wlam0-2} together, we get $w_{\lambda_0}(x^0)=0$ and so
$x^0\in \partial H_{\lambda_0}\cap T_{\lambda_0}$.

Moreover, since $x^k\in \Sigma_{\lambda_k}$ is the interior minimum point, we have
\begin{equation*}
	\nabla w_{\lambda_{k}}(x^k)=0,
\end{equation*}
which implies
\begin{eqnarray}
	\nonumber  0 &=& \frac{\partial }{\partial x_n} w_{\lambda_{k}}(x^k)\\
	\nonumber &=& \frac{\partial }{\partial x_n}\left(
	u(x', 2\lambda_k-x_n)-u(x', x_n)
	\right)|_{x^k}\\
	\label{eq:dudx11} &=&-\frac{\partial u}{\partial x_n}( (x^k)', 2\lambda_k-x^k_n)-\frac{\partial u}{\partial x_n}((x^k)', x^k_n).
\end{eqnarray}
Taking limit $k\to \infty$ in \eqref{eq:dudx11}, we obtain $0=-2\frac{\partial u}{\partial x_n}(x^0)$, i.e.
\begin{equation}
	\label{eq:du0}
	\frac{\partial u}{\partial x_n}(x^0)=0.
\end{equation}

Dividing both side of \eqref{sec:sev} by $\delta_{k}=|\lambda_k-x^k_n|=\lambda_k-x^k_n$ and using Lemma \ref{lem:two}, we deduce that
\begin{equation} \label{cot}
	\qquad 0 > \varlimsup_{k\to \infty}\frac{L_\Delta w_{\lambda_k}(x^k)}{\delta_k}\geq  \varlimsup_{k\to \infty}a(x_n^k)M(\lambda_k,x^k)\frac{w_{\lambda_k}(x^k)}{\delta_k}.
\end{equation}
Noticing that as $k\to \infty$, we have $\delta_k\to 0$, $x^k\to x^0$, $a(x_n^k)M(\lambda_k,x^k)\to a(\lambda_0)M(\lambda_0,x^0)$ and
\begin{eqnarray}
	\label{dw}
	\frac{w_{\lambda_{k}}(x^k)}{\delta_k}&=&\frac{u((x^k)', 2\lambda_k-x^k_n)-u((x^k)', x^k_n)}{|\lambda_k-x^k_n|}\\
	\nonumber	&=& \frac{\frac{\partial u}{\partial x_n}((x^k)', \theta^k)\cdot 2(\lambda_k-x^k_n)}{\lambda_k-x^k_n}\\
	\nonumber	&=& 2 \frac{\partial u}{\partial x_n}((x^k)', \theta^k)\to 2\frac{\partial u}{\partial x_n}(x^0)=0\quad \textrm{as} k\to \infty,
\end{eqnarray}
in which $\theta^k$ is between $2\lambda_k-x^k_n$ and $x^k_n$, and also we have used \eqref{eq:du0}.

Putting \eqref{dw} into the right side of \eqref{cot} will lead to a contradiction.

Therefore, we have $\lambda_0=+\infty.$
\end{proof}

\section{Non-existence of the positive bounded solutions}\label{sec:leq}
In the previous section, we have proved that the positive bounded solutions of \eqref{sec:six} is monotone increasing along $x_n$ direction. Based on this result, in this section, we will prove the non-existence of positive bounded solutions for equation \eqref{sec:six}.

The following maximum principle will be needed in our proof. It has been obtained in \cite{CLL1} (Proposition 4.1).
We include the proof for readers' convenience.
Let $B_1(Re_n)\subset \mathbb{R}^n$ be the unit ball centered at $(0,\dots,0,R)$.

\begin{proposition}[Strong Maximum Principle in a unit ball]
\label{pro:one}
Let $u\in  C(\overline{B_1(Re_n)})$ satisfying
\begin{equation}
	\label{sec:twen}
	\left\{
	\begin{array}{ll}
		L_\Delta u(x)=\left((-\Delta)^L+\rho_n\right)u(x)\geq 0, & x\in B_1(Re_n), \\
		u(x)>0,& x\in \mathbb{R}^n\backslash B_1(Re_n).
	\end{array}
	\right.
\end{equation}
Then $u>0$ in $B_1(Re_n).$
\end{proposition}
\begin{proof}
Without loss of generality, suppose for contradiction that there is $x_0\in B_1(Re_n)$ such that
\begin{equation}\label{def:minp}
u(x_0)=\min\limits_{x\in B_1(Re_n)}u(x)\leq0.
\end{equation}
Denote $A=B_1(x_0)\cap B_1(Re_n)$, $D=\mathbb{R}^n\backslash (B_1(x_0)\cup B_1(Re_n))$.
By using the definition of $(-\Delta)^L$, we have  
\begin{eqnarray*}
(-\Delta)^Lu(x_0)&=& C_nP.V.\int_{B_1(x_0)}\frac{u(x_0)-u(y)}{|x_0-y|^n}dy +C_n\int_{\mathbb{R}^n\backslash B_1(x_0)}\frac{-u(y)}{|x_0-y|^n}dy\\
&=& C_nP.V.\int_{A}\frac{u(x_0)-u(y)}{|x_0-y|^n}dy +C_n\int_{B_1(x_0)\backslash A}\frac{u(x_0)-u(y)}{|x_0-y|^n}dy \\
& &-C_n\int_{B_1(Re_n)\backslash A }\frac{u(y)}{|x_0-y|^n}dy-C_n\int_{D}\frac{u(y)}{|x_0-y|^n}dy.
\end{eqnarray*}
Since $u>0$ in $D$ and in $B_1(x_0)\backslash A$, there holds
\begin{eqnarray}\label{est:maxp}
& &(-\Delta)^Lu(x_0)\\
\nonumber
&<& C_nP.V.\int_{A}\frac{u(x_0)-u(y)}{|x_0-y|^n}dy +C_n\int_{B_1(x_0)\backslash A}\frac{u(x_0)}{|x_0-y|^n}dy \\
\nonumber & &+C_n\int_{B_1(Re_n)\backslash A }\frac{u(x_0)-u(y)}{|x_0-y|^n}dy-C_n\int_{B_1(Re_n)\backslash A }\frac{u(x_0)}{|x_0-y|^n}dy\\
\nonumber &=& C_nP.V.\int_{B_1(Re_n)}\frac{u(x_0)-u(y)}{|x_0-y|^n}dy +C_n u(x_0)\left(\int_{B_1(x_0)\backslash A}\frac{1}{|x_0-y|^n}dy-\int_{B_1(Re_n)\backslash A }\frac{1}{|x_0-y|^n}dy\right)\\
\nonumber &\leq & C_n u(x_0)\left(\int_{B_1(x_0)\backslash A}\frac{1}{|x_0-y|^n}dy-\int_{B_1(Re_n)\backslash A }\frac{1}{|x_0-y|^n}dy\right),
\end{eqnarray}
where we have used \eqref{def:minp}.

Obviously, for those $y\in B_1(x_0)\backslash A$, $|x_0-y|<1$. Also, for $y\in B_1(Re_n)\backslash A$, $|x_0-y|>1.$ 
Hence, we can infer that 
\begin{equation}
	\label{est:vol}
	\int_{B_1(x_0)\backslash A}\frac{1}{|x_0-y|^n}dy-\int_{B_1(Re_n)\backslash A }\frac{1}{|x_0-y|^n}dy
	>|B_1(x_0)\backslash A|-|B_1(Re_n)\backslash A |=0.
\end{equation}

Using the fact that $\rho_n>0$ and putting \eqref{est:vol} into \eqref{est:maxp},
we arrive at $	\left((-\Delta)^L+\rho_n\right)u(x_0)\leq (-\Delta)^Lu(x_0)<0$,
which contradicts with \eqref{sec:twen}. 
\end{proof}

\begin{remark}
We note that the aforementioned maximum principle remains valid on a bounded domain $\Omega$, provided that the condition
$$
\int_{B_1(x)\backslash \Omega}\frac{1}{|x-y|^n}dy-\int_{\Omega\backslash B_1(x) }\frac{1}{|x-y|^n}dy+\rho_n\geq 0, \quad \forall x\in \Omega,
$$
holds, which is precisely Proposition 4.1 in \cite{CLL1}.
Additionally, the result holds for $u(x)$ satisfying
\begin{equation*}
	\left\{
	\begin{array}{ll}
		(-\Delta)^L u(x)\geq 0, & x\in B_1(Re_n), \\
		u(x)>0,& x\in \mathbb{R}^n\backslash B_1(Re_n).
	\end{array}
	\right.
\end{equation*}
The proofs for these assertions follow a similar reasoning as the one presented above.
\end{remark}

We shall now establish the proof of Theorem \ref{result:sum}, following the approach presented in \cite{CLL19}.
Consider $\phi(x)$ as the first eigenfunction associated with $L_\Delta$ in $B_1(Re_n)$, that is, the function defined as follows: 
\begin{equation*}
	\label{*}
	\left\{
	\begin{array}{ll}
		L_\Delta\phi(x)=\lambda_1\phi(x), & x\in B_1(Re_n), \\
		\phi(x)=0, & x\in B_1^c(Re_n).
	\end{array}
	\right.
\end{equation*}
According to Theorem 1.4 and Theorem 1.11 in \cite{CLL1}, we can deduce that $\phi(x)$ is strictly positive in $B_1(Re_n)$, $\lambda_1$ is positive, and $\phi(x)$ belongs to $C\big(\overline{B_1(Re_n)}\big)$.

\begin{proof}[Proof of Theorem \ref{result:sum}]
Assume for contradiction that $u\in L_0\cap C_{loc}^{1,1}(\Omega)\cap C(\bar\Omega)$ is a positive bounded solution of problem \eqref{sec:six}. By Theorem \ref{result:min}, $u(x)$ is monotone increasing in $x_n$ direction in $\Omega$.

It follows from that assumption that $\Omega$ is a Lipschitz coercive epigraph, we know that there is $R_1>0$ so that $B_1(R_1e_n)\subset \Omega$. 
Since $u(x)$ is positive in $\Omega$, 
$$
\xi_0:=\min_{x\in B_1(R_1e_1)}u(x)>0.
$$  
By the fact that $u(x)$ is monotone increasing in $x_n$ direction in $\Omega$, we get for all $R>R_1$, $u(x)>\xi_0$, $\forall x\in B_1(Re_n)$. Since $a(t)$ satisfies (i) and (iv), it is positive somewhere and nondecreasing. Hence, we choose $R>R_1$ sufficiently large such that $a(t)>0$ in $(R-1,+\infty)$.

Denote $m_0=\frac{f(\xi_0)}{\sup\limits_{R^{n}}u}>0.$
It follows from assumptions (i) and (iii) 
\begin{equation*}
L_\Delta u(x)=a(x_n)f(u(x))\geq a(R-1)f(\xi_0)\geq a(R-1)m_0u(x),\quad \forall x \in B_1(Re_n).
\end{equation*}

Taking assumption $(iv)$ into account, we have 
\begin{equation}\label{twy}
	L_\Delta u(x)\geq\lambda_1 u(x),\quad \forall x\in B_1(Re_n).
\end{equation}
for sufficiently large $R$.

Next we construct a auxiliary functions. Define
$$M=\max\limits_{x\in B_1(Re_n)}\frac{\phi}{u}(x),\quad \textrm{and} \quad v(x)=Mu(x)\in C_{loc}^{1,1}\cap L_0.$$
By using \eqref{twy} and a direct calculation,  we have for all $x\in B_1(Re_n),$
\begin{equation*}
	\begin{aligned}
		 L_\Delta v(x)&=M L_\Delta u(x)\geq M\lambda_1u(x)=\left(\max\limits_{x\in B_1(Re_n)}\frac{\phi}{u}(x)\right)\lambda_1u(x)\\
		&\geq \frac{\phi(x)}{u(x)}\lambda_1u(x)=\lambda_1\phi(x)=L_\Delta \phi(x),
	\end{aligned}
\end{equation*}
which implies
\begin{equation}
	\label{eq:comp}
	\left\{
	\begin{array}{ll}
		L_\Delta (v(x)-\phi(x))\geq0, & x\in B_1(Re_n), \\
		v(x)-\phi(x)>0,& x\in B_1^c(Re_n).
	\end{array}
	\right.
\end{equation}
Applying Proposition \ref{pro:one} to problem \eqref{eq:comp}, we  have
$$v(x)>\phi(x),\quad \forall x\in B_1(Re_n).$$
This contradicts the definition of $v,$ because at a maximum point $x^0,$ we have 
$$v(x^0)=\frac{\phi(x^0)}{u(x^0)}u(x^0)=\phi(x^0).$$
Therefore, equation \eqref{sec:six} does not possess any positive solution, and hence we complete the proof of Theorem \ref{result:sum}.
\end{proof}

\end{document}